\documentclass{article}
\usepackage[a4paper, total={6in, 8in}]{geometry}
\usepackage{cite}
\usepackage{amsmath}
\usepackage{amsfonts}
\usepackage{amssymb}
\usepackage{amsthm}
\usepackage{authblk}

\title{An invariant version of the little Grothendieck theorem for Sobolev spaces}
\author[1]{K. Kazaniecki}
\author[2]{P. Pakosz}
\author[1]{M. Wojciechowski}
\affil[1]{Institute of Mathematics Polish Academy of Sciences}
\affil[2]{University of Warsaw}

\newtheorem{tw}{Theorem}
\newtheorem{Cj}{Conjecture}

\newtheorem{df}{Definition}

\newtheorem{lem}{Lemma}
\newtheorem{rem}{Remark}

\newcommand{\Z}{\mathbb{Z}}

\newcommand{\C}{\mathbb{C}}
\newcommand{\T}{\mathbb{T}}
\newcommand{\krata}{\mathbb{Z}^2 \setminus \{ (0,0) \}}
\newcommand{\dotp}[2]{2 \pi i \langle #1 , #2 \rangle}
\newcommand{\dotprod}[2]{\langle #1 , #2 \rangle}

\begin{document}
\maketitle

\begin{abstract} We prove that every Hilbert space operator which factorizes invariantly through Sobolev space $W^{1}_{1}(\mathbb{T}^d)$ belongs to some non-trivial Schatten class.
\end{abstract}

\section{Introduction}
One of the most important theorems of the operator ideals theory proved by Grothendieck in 1953 \cite{MR94682} states that any bounded linear operator $T:L_1 \rightarrow L_2$ is absolutely summing ( for definitions and notation used in this paper see \cite{Diestel},\cite{Wojtaszczyk1991}). This theorem, beginning with the Lindenstrauss-Pe\l czy\'{n}ski (cf. \cite{LINPEL}) paper, inspired a great development of Banach space geometry and many other fields (cf. \cite{Pisier2012}) and still attracts wide attention (cf. \cite{Braverman2011}, \cite{MR3266999}, \cite{LUST07}, \cite{Blei14}).

A weaker (and simpler) version of Grothendieck's theorem claims that such operators are 2-absolutely summing. This weaker statement is equivalent to the property that every operator between Hilbert spaces which factorizes through $L_1$ is a Hilbert-Schmidt operator (cf.\cite{Garling2015}, Prop 16.3.2). As it was observed by Kislyakov (cf. \cite{Kis75}), one cannot replace $L_1$ by the Sobolev space $W^{1}_{1}$ in this theorem. Indeed, the classical embedding operator of $W^{1}_{1}(\T^2)$ to $L_2(\T^2)$ is not 2-abolutely summing.
 However, as proved in \cite{Wojciechowski1997}, it is $(p,1)$-summing for every $p>1$. This suggests the following conjecture. 

\begin{Cj}\label{CJ1}
Any operator between Hilbert spaces which factorizes through the Sobolev space $W^{1}_{1}$ belongs to some non-trivial Schatten class.
\end{Cj}
Not only are we unable to prove the conjecture, but we do not even know if there exists an infinitely dimensional complemented subspace of $W^{1}_{1}$ which is isomorphic to a Hilbert space. However, as is was proved in (cf. \cite{Pelczynski1992}), there are no such spaces which are translation invariant. This suggests that assuming some additional structure may help to verify the conjecture and motivate us to consider its special case - translation invariant operators. We introduce the following definition  
\begin{df}
Let $X(\T^d), Y(\T^d), Z(\T^d)$ be translation invariant spaces on the $d$-dimensional torus. We say that a bounded linear operator $T:X(\T^d)\rightarrow Y(\T^d)$ admits invariant factorization through $Z(\T^d)$ provided there exist translation invariant bounded linear operators $A:X(\T^d)\rightarrow Z(\T^d)$ and $B:Z(\T^d)\rightarrow Y(\T^d)$ such that $T= B\circ A$. 
\end{df}
Under this restriction we were able to prove our conjecture. The main result of this paper is 
\begin{tw}\label{invariantnafaktor}
 Let  $T: L_2(\mathbb{T}^d)\rightarrow L_2(\mathbb{T}^d)$ be a bounded linear operator which admits an invariant factorization through $W^{1}_{1}(\mathbb{T}^d)$. Then $\sigma_{p}(T)<\infty$ for every $p>2d+4$, where $\sigma_p$ denotes the $p$-th Schatten norm.
\end{tw}
The main theorem is a direct consequence of the following estimate on the growth of the coefficients of Fourier multiplier operators.

\begin{tw}\label{wniosek}
 Let $T \colon W^{1}_{1}(\mathbb{T}^d) \to L_2(\mathbb{T}^d)$  be a translation invariant operator such that
	\begin{displaymath}
		\widehat{T(f)}(n) = \lambda_n \cdot \widehat{f}(n) .
	\end{displaymath}
	Then for any $\varepsilon > 0$ the following inequality is satisfied
	\begin{displaymath}
		\sum_{n \in \Z^d\backslash\{0\}} \left ( \frac{|\lambda_n|}{|n|_2} \right )^{2d+4+\varepsilon} < \infty .
	\end{displaymath}
\end{tw}

We show now how the main result follows from Theorem \ref{wniosek}.
Let $T$ given by $T(e^{i\langle n,t\rangle})=\lambda_n e^{i\langle n,t\rangle}$ factorizes by $A,B$ where $A:W^{1}_{1}(\T^d)\rightarrow L_2(\T^d)$ is given by $A(e^{i\langle n,t\rangle})=\alpha_n e^{i\langle n,t\rangle}$ and $B:L_2(\T^d)\rightarrow W^{1}_{1}(\T^d)$ is given by $B(e^{i\langle n,t\rangle})=\beta_n e^{i\langle n,t\rangle}$ for $n\in\mathbb{Z}^d$. Then using an obvious estimate $|\beta_n||n|_2\leq \|B\|$ we get 
\[
\sum |\lambda_n|^p= \sum \left|\alpha_n \cdot \beta_n\right|^p\leq\|B\|^p \sum\left(\frac{\alpha_n}{|n|_2}\right)^p \leq C \|A\|^p\|B\|^p
\]
Theorem \ref{wniosek} is a special case of the following result.
\begin{tw}\label{r:twGlowne}
Let $1<p\leq 2$ and $T \colon W^{1}_{1}(\mathbb{T}^d) \to L_{p}(\mathbb{T}^d)$ be a translation invariant operator such that
	\begin{displaymath}
		\widehat{T(f)}(n) = \lambda_n \cdot \widehat{f}(n) .
	\end{displaymath}
	Then for any $\varepsilon > 0$, the following inequality is satisfied
	\begin{displaymath}
		\sum_{n \in \mathbb{Z}^d\backslash\{0\}} \left ( \frac{|\lambda_n|}{|n|_2} \right )^{\frac{p}{p-1}+\frac{p}{p-1}(d+1)+\varepsilon} < \infty .
	\end{displaymath}
\end{tw}

 The rest of the paper is devoted to the proof of Theorem \ref{r:twGlowne}. 
\begin{rem} We don't know whether the exponent in Theorem \ref{invariantnafaktor} is sharp. The most prominent example of an operator which is a subject of Theorem \ref{invariantnafaktor} is the classical Sobolev embedding operator $T
:W^{1}_{2}(\mathbb{T}^2) \to L_2(\mathbb{T}^2)$. In this case we have $\lambda_{n} = 1$ and $\sigma_p(T)<\infty$ for $p>2$ which is much stronger than the statement of Theorem \ref{invariantnafaktor} (which gives $p > 8$). Possibly the exponent in the \mbox{Theorem \ref{invariantnafaktor}} could be improved.
\end{rem}
\begin{rem} As a matter of fact the result of this paper is a subject of harmonic analysis and it concerns Fourier multipliers. We only indicated an interpretation for it in terms of operator ideals and used this interpretation to support Conjecture \ref{CJ1}. Note that invariant little Grothendieck theorem for $L_1$ is an obvious well known property ( see \cite{Hormander} Theorem 1.4).
\end{rem} 

\begin{rem} In the case when $X$ and $Y$ are Hilbert spaces we can formulate the definition of invariant factorization in a more abstract way. For a compact abelian group $G$ and an invariant function space $Z(G)$  we say that an operator $T: H_1 \rightarrow H_2$ factorizes invariantly through $Z(G)$ if there are orthonormal bases $\{h_{n,1}\}$ of $H_1$ and $\{h_{n,2}\}$ of $H_2$ and operators $A: H_1 \rightarrow Z(G)$ and $B : Z(G)\rightarrow H_2$ such that $T=B\circ A$ and  $A(h_{n,1})\in \operatorname{span}\gamma_n$ and $B(\gamma_{n})\in \operatorname{span} h_{n,2}$ for $n=1,2,\dots$, where $(\gamma_n)$ is an enumeration of characters of the group $G$.
\end{rem}
\begin{Cj}
If $Z(G)$ has no complemented, invariant infinitely dimensional subspaces isomorphic to a Hilbert space, then any Hilbert space operator $T$ which factorizes invariantly through $Z(G)$ belongs to some nontrivial Schatten class.
\end{Cj}    

\section{Summability of the multiplier}

For any $k \in \mathbb{N}$ we denote by \mbox{$R_k = \{ (n_1,\ldots,n_d) \in \mathbb{Z}^d : 3^k \leq \max_{i}\{|x_i|\} < 3^{k+1} \}$} a k-th triadic ring.  For $n\in \mathbb{Z}^d$  by $n^{(i)}$ we denote the i-th coordinate of $n$ and by $|n|_2$ it's euclidean norm. Observe that for $n\in R_k$ we have
\begin{equation}\label{euck}
3^k \leq |n|_2 \leq \sqrt{d} 3^{k+1}
\end{equation}
In order to prove Theorem \ref{r:twGlowne} we will use two auxiliary lemmas on the growth of $\lambda_n$ on triadic rings.
In the first lemma we control the behavior of the sequence $\lambda_n$ on  a single triadic ring.

\begin{lem}\label{krok1}
	There exists a constant $C>0$ independent of $k$, such that for any $k \in \mathbb{N}$ 
	\begin{displaymath}
		\sum_{n \in R_k} \left ( \frac{|\lambda_n|}{|n|_2} \right )^{\frac{p}{p-1}} < C .
	\end{displaymath}
\end{lem}
The second lemma provides an estimate on the growth of the sequence of maximal elements of $\lambda_n$ in the triadic rings.
\begin{lem}\label{krok2}
	For every $\varepsilon > 0$
	\begin{displaymath}
		\sum_{k \in \mathbb{N}} \max_{n \in R_k} \left ( \frac{|\lambda_n|}{|n|_2} \right )^{\frac{p}{p-1}(d+1)+\varepsilon} < \infty .
	\end{displaymath}
\end{lem}

We postpone the proofs of lemmas to Section \ref{dw1} and Section \ref{dw2}. Now we use them to prove the main theorem. 
\paragraph{\textit{Proof of the theorem \ref{r:twGlowne}}} For $n \in R_k$ we have the following estimate
\begin{equation*}
\begin{split}
    \left ( \frac{|\lambda_n|}{| n |_2} \right )^{\frac{p}{p-1}+\frac{p}{p-1}(d+1) + \varepsilon} &\!\!\!= 
    \left ( \frac{|\lambda_n|}{| n |_2} \right )^{\frac{p}{p-1}} \cdot \left ( \frac{|\lambda_n|}{| n |_2} \right )^{\frac{p}{p-1}(d+1) + \varepsilon} 
    \\&\leq \left ( \frac{|\lambda_n|}{| n |_2} \right )^{\frac{p}{p-1}} \cdot \max_{n \in R_k} \left ( \frac{|\lambda_n|}{| n |_2} \right )^{\frac{p}{p-1}(d+1) + \varepsilon} .
\end{split}    
\end{equation*}
From Lemma \ref{krok1} and the above estimate we get
\[
\begin{split}
    \sum_{n \in \mathbb{Z}^d\backslash\{0\}} \left ( \frac{|\lambda_n|}{|n|_2} \right )^{\frac{p}{p-1}(d+2)+\varepsilon}\!\!\!\!\! &=
    \sum_{k \in \mathbb{N}} \sum_{n \in R_k} \left(\frac{|\lambda_n|}{| n |_2} \right)^{\frac{p}{p-1}+\frac{p}{p-1}(d+1)+\varepsilon} \\
    & \leq \sum_{k \in \mathbb{N}} \sum_{n \in R_k} \left(\frac{|\lambda_n|}{| n |_2} \right)^{\frac{p}{p-1}} \cdot \max_{n \in R_k} \left( \frac{|\lambda_n|}{| n |_2} \right)^{\frac{p}{p-1}(d+1) + \varepsilon} \\
    & < C \cdot \sum_{k \in \mathbb{N}} \max_{n \in R_k} \left( \frac{|\lambda_n|}{| n |_2} \right)^{\frac{p}{p-1}(d+1)+\varepsilon} 
    \end{split}    
\]
By Lemma \ref{krok2}, the right hand side of the above inequality is finite. Hence
\begin{equation*}
\begin{split}
    \sum_{n \in \krata} \left ( \frac{|\lambda_n|}{|n|_2} \right )^{\frac{p}{p-1}+\frac{p}{p-1}(d+1)+\varepsilon} < \infty ,
\end{split}  
\end{equation*}

\section{Proof of Lemma \ref{krok1}}\label{dw1}
The proof of Lemma \ref{krok1} is based on the well known properties of Fejer's kernel. It is standard, however we present it here for reader's convenience.\\
We will denote by  $K_n$ the classical Fejer's kernel: 
\[    \widehat{K_n}(k) = \left\{\begin{array}{cc} 1-\frac{|k|}{n} \quad &\textrm{for} \quad |k| \leq n \\ 0 \textrm{ } \quad &\quad \textrm{in other case}.\end{array}\right.
\]
For a fixed $k\in \mathbb{N}$ we define $\phi(x_1,\ldots,x_d) \colon \T^d \to \mathbb{C}$ by the formula
\begin{displaymath}
	\widehat{\phi} \left(m_1,\dots,m_d \right) = \Pi_{j=1}^{d}\widehat{K_{3^{k+2}}} \left(m_j \right).
\end{displaymath}
Since $K_n$ is a trigonometric polynomial of degree $n$, by the classical Bernstein's inequality we have
\begin{equation*}
\begin{split}
    \left \| \frac{\partial}{\partial x_j} \phi \right \|_1  &= \left \| \frac{\partial}{\partial y} K_{3^{k+2}}(y) \right \|_1 \leq 3^{k+2} \| K_{3^{k+2}}(y)\|_1.
\end{split}    
\end{equation*}
Therefore 
\begin{displaymath}
    \| \phi \|_{1,1}:= \|\phi\|_1 + \sum_{j=1}^{d} \bigg\|\frac{\partial}{\partial x_j} f\bigg\|_{1} \leq 1 + d \cdot 3^{k+2}.
\end{displaymath}

Let us observe that for  $m\in R_k$ we have $\left | \widehat{\phi}(m) \right | \geq \left(\frac{2}{3}\right)^d $. Indeed, $|m_j| < 3^{k+1}$ and $$1-\frac{|m_j|}{3^{k+2}} > \frac{2}{3}.$$ By the Hausdorff-Young inequality, we get
\begin{displaymath}
	\left \| T \phi \right \|_{p} \geq \left \| \lambda \cdot \widehat{\phi} \right \|_{\frac{p}{p-1}}
	\geq \left(\frac{2}{3}\right)^{d} \left(\sum_{n \in R_k} | \lambda_n|^{\frac{p}{p-1} }\right)^{\frac{p-1}{p} } .
\end{displaymath}

Combining the above estimates with \eqref{euck} completes the proof.
\begin{equation*}
\begin{split}
    \sum_{n \in R_k} \left( \frac{|\lambda_n|}{|n|_2} \right)^{\frac{p}{p-1}} & \leq 
3^{-\frac{pk}{p-1}} \cdot \sum_{n \in R_k} \left | \lambda_n  \right |^{\frac{p}{p-1}} 
\\    & \leq 
3^{-\frac{pk}{p-1}} \cdot \left( \frac{3}{2} \right)^{\frac{dp}{p-1}} \cdot \left \| T \phi \right \|_{p}^{\frac{p}{p-1}} 
\\    & \leq 
3^{-\frac{pk}{p-1}}\cdot \left( \frac{3}{2} \right)^{\frac{dp}{p-1}} \cdot \left \| T \right \|^{\frac{p}{p-1}} \cdot \left \| \phi \right \|^{\frac{p}{p-1}}_{1,1} \\
    & \leq 
3^{-\frac{pk}{p-1}} \cdot \left( \frac{3}{2} \right)^{\frac{dp}{p-1}} \cdot \| T \|^{\frac{p}{p-1}} \cdot \left( 1 + d  \cdot 3^{k+2} \right)^{\frac{p}{p-1}} \\
    & \leq C(p,d)\| T \|^{\frac{p}{p-1}} .    
\end{split}    
\end{equation*}

\section{Proof of Lemma \ref{krok2}}\label{dw2}
For any fixed sequence $\{n_i\}_{i=1}^{N}\subset \mathbb{Z}^{d}\backslash \{0\}$ such that $\frac{|n_{i+1}|_2}{|n_i|_2}>3$, we define a finite Riesz product corresponding to sequence $\{n_j\}_{j=1}^{N}$, a trigonometric polynomial $R \colon \mathbb{T}^d \to \mathbb{R}$ given by a formula
\begin{displaymath}
	R(t) = \prod_{j=1}^N \left( 1+\cos \left( 2 \pi \dotprod{n_j}{t} \right) \right).
\end{displaymath}
Let $X=\{-1,0,1\}^N$. For any $\xi \in X$ we define $L(\xi)$ as a number of non-zero coordinates of $\xi$ i.e. $L(\xi)= \sum_{j=1}^{N} |\xi_j|$ and
 $M(\xi) = \sum_{j=1}^N \xi_j \cdot n_j$. One can easily check that
\begin{equation}\label{eq:rozwiniecie}
	R(t) = \sum_{\xi \in X} \frac{1}{2^{L(\xi)}} e^{\dotp{M(\xi)}{t}} .
\end{equation}
and
\begin{equation}\label{riesz}
    \| R(t) \|_1 = 1.
\end{equation}
We introduce the auxiliary  notion of the growth of the sequence $\{n_j\}_{j=1}^{N}$.
\begin{df}
	Let $\alpha > 1$. We will call a sequence $n_1, \ldots, n_N \in \mathbb{Z}^d\backslash \{0\}$ an $\alpha$-sparse if $\frac{|n_{j+1}|_2}{|n_j|_2} \geq 3^{\alpha}$ for every $j = 1, \ldots, N-1$.
\end{df}
For $j\in{1,\ldots, d}$ we define set
\[
A_j=\{n\in\mathbb{Z}^d\backslash\{0\}: j\mbox{ is a smallest index such that }|n_j|=\max_{k\in\{1,\cdots, d\}} |n_k|\}
\]
Clearly the sets $A_j$ are pairwise disjoint and 
\[
\mathbb{Z}^d\backslash\{0\}=\bigcup_{j=1}^{d} A_j
\]
Let $j\in{1,\ldots,d}$ then for $n\in A_j$ and $k\in\{1,\ldots,d\}\backslash\{j\}$ we have $\frac{n_k}{n_j} \in [-1,1]$. Fix a large number $N$. We want to further subdivide sets $A_j$ into pieces to control the value of quotients $n_k\slash n_j$ up to $1\slash N$ for $n$ in a single piece. For such $N$ we define sets  $\tilde{A}_{j,a}:=\tilde{A}_{j,a}(N) \subseteq \mathbb{Z}^d\backslash \{0\}$ for $j\in\{1,\ldots, d\}$, \mbox{$a\in \{1,\ldots,N\}^{d} \cup\{a_j=1\}$}
\begin{displaymath}
	\tilde{A}_{j,a}= \left\{ n\in \mathbb{Z}^d\backslash \{0\} \colon\!\!\!  \left| \frac{n_k}{n_j}-\left(\frac{2a_k-1}{N}-1\right)\right|\leq \frac{1}{N},\;\,  k=1,\cdots,d \right\}.
\end{displaymath}
One can check that
\[
A_j=\cup_{a}\tilde{A}_{j,a} 
\]
Since the sets $\tilde{A}_{j,a}$ are no longer pairwise disjoint, we choose symmetric and pairwise disjoint sets $A_{j,a}:=A_{j,a}(N)$ such that $A_{j,a}\subset\tilde{A}_{j,a}$ and
\[
\cup_{a} A_{j,a}=A_{j}.
\]
Obviously
\[
\cup_{j,a} A_{j,a}=\cup_{j} A_j = \mathbb{Z}^d\backslash \{0\}.
\]
\begin{df}\label{sektory}
We call the set $A_{j,a}$ an $N$-sector.
\end{df}

\begin{rem}
We choose sets $A_{j,a}$ in such way that vectors from fixed set $A_{j,a}$ almost point in the same direction, up to some error which depends on $N$. This allows us to construct in next section a test function $\phi$, which behaves like a function whose Fourier spectrum is contained in a fixed line. Note that the function $\psi$ whose Fourier spectrum is contained in a fixed line satisfies equation
\[
\nabla \psi = v \frac{\partial}{\partial x_j} \psi  
\]  
for some fixed $v\in \mathbb{R}^{d}$. In fact what we need from the sets $A_{j,a}$ 
\begin{enumerate}
\item to be symmetric.
\item to be pairwise disjoint.
\item $\cup_{j,a} A_{j,a}= \mathbb{Z}^d\backslash \{0\}$.
\item $|\tan\measuredangle(v_{j,\alpha},n)|<\frac{C}{N}$ for all $n\in A_{j,a}$, fixed $C>0$ and fixed $v_{j,\alpha}\in A_{j,\alpha}$.
\item For $n\in A_{j,\alpha}$ the $j$-th coordinate dominates other coordinates.
\end{enumerate}
\end{rem}

\paragraph{\textit{Construction of the test function.}} Let $N$ be a fixed natural number and a sequence \mbox{$\{n_j\}_{j=1}^{N}\subset \mathbb{Z}^d\backslash \{0\}$} be an $N$-sparse sequence contained in the single \mbox{$N$-sector} $A_{j_0,a}$. Let $R(t)$  be a finite Riesz product corresponding to the sequence $n_1, \ldots, n_N$. We define a function  $\phi = \phi_{N; n_1, \ldots, n_N} \colon \T^d \to \C$ by $\int_{\T^d} \phi(t) d \mu(t) = 0$ and 
\begin{equation}\label{df: phi}
	\frac{\partial}{\partial x_{j_0}} \phi = R(t) - 1.
\end{equation}
We will estimate the Sobolev norm of the function $\phi$.
Without loss of generality we can assume that the sequence is a subset of $A_{1,a}$. From the triangle inequality and  \ref{riesz} we get

\begin{displaymath}
	\left \| \frac{\partial}{\partial x_1} \phi \right \|_1 = \left \| R(t) - 1 \right \|_1 \leq \left \| R(t) \right \|_1 + \left \| 1 \right \|_1 = 2 .
\end{displaymath}
We will estimate the remaining derivatives.

\begin{lem}\label{lemgl}
    There exists a constant $C>0$, such that for any number $N$, any $a \in [0,N]^d$ and any  $N$-sparse sequence  $n_1, \ldots, n_N \in \mathbb{Z}^d\backslash\{0\}$ contained in the $N$-sector $A_{1,a}$ holds
    \begin{displaymath}
        \left \| \frac{\partial}{\partial x_j} \phi \right \|_1 < C  \qquad\forall\; j\in\{2,\ldots,d\}.
    \end{displaymath}
\end{lem}

\begin{proof}Note that
\[
	\phi = \sum_{\xi \in X \setminus \{ 0 \} } \frac{1}{2^{L(\xi)}} \frac{1}{M(\xi)^{(1)}} e^{\dotp{M(\xi)}{t}} .
\]
and
\begin{equation}\label{eq:wzor}
	\frac{\partial}{\partial x_j} \phi = 
	\sum_{\xi \in X \setminus \{ 0 \}} \frac{1}{2^{L(\xi)}} \frac{M(\xi)^{(j)}}{M(\xi)^{(1)}} e^{\dotp{M(\xi)}{t}} .
\end{equation}
Let 
\[
X_l:=\{\xi\in X: \forall_{k>l}\; \xi_k=0 \},
\]
We define $\psi_l$ by the formula
$$
    \psi_l = \sum_{\xi \in X_{l-1}} \frac{1}{2^{L(f)}} e^{\dotp{M(\xi)}{t}} .   
$$
The function $\psi_l$ is just a finite Riesz product corresponding to the sequence $n_1,\ldots , n_l$. Hence $\| \psi_l \|_1 = 1$.
We can rewrite the equation \eqref{eq:rozwiniecie} in the following way
\begin{equation}\label{eq:tozsamosc}
\begin{split}
   R ( t )-1& =\sum_{\xi \in X\setminus \{ 0 \}} \frac{1}{2^{L(\xi)}} e^{\dotp{M(\xi)}{t}}
   =\sum_{l=1}^{n}\sum_{\xi \in X_l} \frac{1}{2^{L(\xi)}} e^{\dotp{M(\xi)}{t}}
   \\
   &=\sum_{l=1}^{N} \frac{1}{2} \left(e^{\dotp{n_l}{t}} + e^{\dotp{-n_l}{t}} \right) \psi_l .     
   \end{split}
\end{equation}
We define an auxiliary function $H_l(\xi)$ for $\xi \in X_{l-1}$
\begin{equation*}
\begin{split}
    H_l \left( \xi \right) = &\frac{1}{2}e^{\dotp{n_l}{t}} \left( \frac{M(\xi)^{(j)}+n_l^{(j)}}{M(\xi)^{(1)}+n_l^{(1)}} - \frac{n_l^{(j)}}{n_l^{(1)}} \right)
    \\ &+ \frac{1}{2} e^{\dotp{-n_l}{t}} \left( \frac{M(\xi)^{(j)} - n_l^{(j)}}{M(\xi)^{(1)} - n_l^{(1)}} - \frac{n_l^{(j)}}{n_l^{(1)}} \right) .
\end{split}
\end{equation*}
Similarly as in \eqref{eq:tozsamosc} we can rewrite equation \eqref{eq:wzor} in terms of  $H_l(\xi)$ and $\psi_l$
\begin{equation}\label{eq:gl}
\begin{split}
     \frac{\partial}{\partial x_j} \phi = &\sum_{l=1}^N \sum_{f \in X_{l-1}} \frac{1}{2^{L(\xi)}} H_l \left( \xi \right) e^{\dotp{M(\xi)}{t}} \\ &+ 
    \sum_{l=1}^N \frac{n_l^{(j)}}{n_l^{(1)}} \frac{1}{2} \left( e^{\dotp{n_l}{t}}+e^{\dotp{-n_l}{t}} \right) \psi_l .  
\end{split}
\end{equation}
In order to estimate the norm of the first term on the right hand side we need the following lemma.
\begin{lem} \label{wspol}
    There exists a constant $C>0$ independent of $N$ and a sequence $\{n_k\}$ such that for any $\xi \in X_l$ we have
    $\| H_l(\xi) \|_1 \leq \frac{C'}{3^N}$.
\end{lem}

Assuming Lemma \ref{wspol} we get the following bound:

\begin{equation}\label{eq:lemacik}
\begin{split}
    \left \| \sum_{l=1}^N \sum_{\xi \in X_l} \frac{1}{2^{L(\xi)}} H_l ( \xi ) e^{\dotp{M(\xi)}{t}} \right \|_1   & \leq
    \sum_{l=1}^N \sum_{\xi \in X_l} \left \| H_l(\xi) \right \|_1 
    \\& \leq  \sum_{l=1}^N \sum_{\xi \in X_l} \frac{C'}{3^N} \leq   \sum_{l=1}^N \frac{C'' \cdot 3^{l-1}}{3^N}   \leq C'' .    
\end{split} 
\end{equation}

Now we estimate the second term on the right hand side of \eqref{eq:gl}. Let $\theta = -1 + \frac{2a_j-1}{N}$,  where $\{n_k\} \subset A_{i,a}$; obviously $\left| \theta \right| \leq 1$. Moreover, if  $ n\in A_{1,a} $ then $\left| \frac{n^{(j)}}{n^{(1)}} - \theta \right| \leq \frac{1}{N}$. From the triangle inequality and \eqref{eq:tozsamosc} we get that 
\begin{equation*}\label{jedenkierunek}
\begin{split}
   \bigg\| \sum_{l=1}^N&\frac{n_l^{(j)}}{n_l^{(1)}} \frac{ e^{\dotp{n_l}{t}}+e^{\dotp{-n_l}{t}}}{2}  \psi_{l} \bigg\|_1  \\
    & \qquad \quad \leq \left \| \sum_{l=1}^N \left( \frac{n_l^{(j)}}{n_l^{(1)}} - \theta_i \right) \frac{e^{\dotp{n_l}{t}}+e^{\dotp{-n_l}{t}}}{2}  \psi_l \right \|_1
    \\
    &\qquad\qquad\quad + \left | \theta_i \right| \cdot \left \|  \sum_{l=1}^N \frac{e^{\dotp{n_l}{t}}+e^{\dotp{-n_l}{t}} }{2} \psi_l \right \|_1 \\
    & \qquad\quad \leq \frac{1}{N} \sum_{l=1}^N  \left \| \frac{e^{\dotp{n_l}{t}}+e^{\dotp{-n_l}{t}}}{2} \psi_l \right \|_1 +  \left \| R(t)-1 \right \|_1 \\
    & \qquad \quad  \leq 1+2 = 3 .
\end{split}    
\end{equation*}
This together with \eqref{eq:lemacik} implies Lemma \ref{lemgl}.
\end{proof} 
It remains to prove Lemma \ref{wspol}.
\begin{proof} Since the sequence $\{n_k\}$ is $N$-sparse we know that for $k \in \{1, \ldots, l \}$ we have
$$
    \max \left( \left|n_{l-j}^{(1)} \right|,\left| n_{l-j}^{(j)} \right| \right) \leq \left| n_{l-k} \right|_{2} \leq 3^{-kN} \left| n_l \right|_{2} .
$$
Hence for $\xi \in X_l$ from the triangle inequality we have the following bounds
$$
    \max \left( \left| M(\xi)^{(1)} \right|,\left| M(\xi)^{(j)} \right| \right) \leq \left| n_l \right|_{2} \cdot \left( \frac{1}{3^N} + \frac{1}{3^{2N}} + \ldots \right) = \left| n_l \right|_{2} \frac{3}{2\cdot 3^{N}}.
$$
Since $n_i\in A_{1,a}$, we know that for $1 \leq i \leq l$ we have $\left| n_i^{(1)} \right| \geq \left| n_i^{(j)} \right|$  and $\sqrt{d} \cdot \left| n_i^{(1)} \right| \geq \left| n_i \right|_{2}$.
From the triangle inequality
\begin{equation*}
\begin{split}
    \left| \frac{n_l^{(j)} \pm M(\xi)^{(j)} }{n_l^{(1)} \pm M(\xi)^{(1)}} - \frac{n_l^{(j)}}{n_l^{(1)}} \right| &=
    \left| \frac{n_l^{(1)} M(\xi)^{(j)} - n_l^{(j)} M(\xi)^{(1)}}{n_l^{(1)} \left( n_l^{(1)} \pm M(\xi)^{(1)} \right) } \right| 
    \\&\leq C \frac{\left| n_l^{(1)} M(\xi)^{(j)} - n_l^{(j)} M(\xi)^{(1)} \right|}{ \left|n_l \right|_{2}^2} \\ &\leq  \frac{C'}{ 3^N}\frac{ \left| n_l \right|_{2}^2}{\left| n_l \right|_{2}^2} = \frac{C'}{3^N}.
\end{split}    
\end{equation*}
Therefore we get
$$
    \left \| H_l(\xi) \right \|_1 \leq \frac{1}{2} \cdot 2 \cdot \left| \frac{n_l^{(j)} \pm M(\xi)^{(j)} }{n_l^{(1)} \pm M(\xi)^{(1)}} - \frac{n_l^{(j)}}{n_l^{(1)}} \right| \leq
    \frac{C'}{3^N} .
$$
\end{proof}
From Lemma \ref{lemgl} and the Poincare inequality we deduce a bound on a $W^{1}_{1}$ norm of the function $\phi$
\begin{equation}\label{szac}
\|\phi\|_{1,1}\leq C \|\nabla \phi\|_{1}\leq C_1,
\end{equation}
where the constant $C_1$ depends only on $d$. 
This estimate is crucial in the proof of Lemma \ref{krok2}.

\begin{lem}\label{pre-krok2}
    There exists a constant $K>0$ independent of $N$ such that for any $N$-sparse sequence $n_1, \ldots, n_N \in \mathbb{Z}^d\backslash\{0\}$ which is a subset of single N-sector we have
$$
    \sum_{i=1}^N \left(\frac{|\lambda_{n_i} |}{|n_i|_2} \right)^{\frac{p}{p-1}} \leq K .
$$
\end{lem}

\begin{proof} Let $\phi$ be defined as in \eqref{df: phi}. Recall that 
$$
    \phi = \sum_{\xi \in X \setminus \{ 0 \}} \frac{1}{2^{L(\xi)}} \frac{1}{M(\xi)^{(1)}} e^{\dotp{M(\xi)}{t}} .
$$
Hence
$$
    T \phi = \sum_{\xi \in X \setminus \{ 0 \}} \lambda_{M(\xi)} \frac{1}{2^{L(\xi)}} \frac{1}{M(\xi)^{(1)}} e^{\dotp{M(\xi)}{t}} .
$$
From the Hausdorff-Young inequality 
$$
    \left \| T \phi \right \|_p^p \geq \left(\sum_{\xi \in X \setminus \{ 0 \}} \left| \lambda_{M(\xi)} \frac{1}{2^{L(\xi)}} \frac{1}{M(\xi)^{(1)}} \right|^{\frac{p}{p-1}}\right)^{p-1} .
$$
We estimate the right hand side summing only over $\xi$ with $L(\xi)=1$. We get
$$
    \left \| T \phi \right \|_p^p \geq C \left(\sum_{i=1}^N \left| \frac{\lambda_{n_i}}{n_i^{(1)}} \right|^{\frac{p}{p-1}}\right)^{p-1} =
    C \left(\sum_{i=1}^N \left| \frac{\lambda_{n_i}}{\left|n_i\right|_2} \cdot \frac{\left|n_i\right|_2}{n_i^{(1)}} \right|^{\frac{p}{p-1}}\right)^{p-1} .
$$
Obviously $\left| n_i \right|_2 \geq \left| n_i^{(1)} \right|$. Hence
$$
    \left \| T \phi \right \|_p^p \geq 
    C \left(\sum_{i=1}^N \left( \frac{|\lambda_{n_i} |}{|n_i|_2} \right)^{\frac{p}{p-1}} \right)^{p-1}.
$$
Boundedness of $T$ and the inequality \eqref{szac} yields the existence of a constant $C>0$ independent of $N$ such that
\[
    \sum_{i=1}^N \left( \frac{|\lambda_{n_i}|}{|n_i|_2} \right)^{\frac{p}{p-1}} \leq C \left \| T \right \|^{\frac{p}{p-1}} .
\]
\end{proof}
We will use the following simple property of the sum of monotonic sequences.
\begin{lem}\label{lema1}
	Let $\{b_j\}_{j=1}^{\infty}$ be a non-negative, non-decreasing  sequence such that 
	\begin{displaymath}
		\sum_{j=1}^N b_j \leq O( N^{\alpha}),
	\end{displaymath}
	where $0<\alpha<1$. Then the sequence $\{b_j\}\in \ell_{q}$ for any $q>\frac{1}{1-\alpha}$.
\end{lem}
\begin{proof} Since the sequence is non-decreasing we have 
\begin{displaymath}
    N \cdot a_N \leq \sum_{j=1}^N a_j \leq C \cdot N^{\alpha} .
\end{displaymath}
Therefore $a_N \leq C N^{\alpha-1}$. 
\end{proof}
The next lemma will be used only to justify the existence of non-increasing rearrangement of the sequence.
\begin{lem}\label{lema2}
    If $\{n_i\} $ is a sequence of points in $\mathbb{Z}^d\backslash\{0\}$ such that $\lim_{i \to \infty} |n_i|_2 = \infty$ then
    $$
        \lim_{i \to \infty} \frac{| \lambda_{n_i} |}{|n_i|_2} = 0 .
    $$
\end{lem}
\begin{proof} Assume that there is a sequence $n_i \to \infty$ such that $\frac{|\lambda_{n_i}|}{|n_i|_2} > c>0$. 
We fix $N$ and we divide  $\mathbb{Z}^d\backslash\{0\}$  into finite number of $N$-sectors (see Definition \ref{sektory}). There exists an infinite subsequence $\{n_{i}\}$ contained in one of them. Passing again to the subsequence we can assume that $\{n_i\}$ is $N$-sparse. From the assumptions on $n_i$ we have
$$
    \sum_{i=1}^N \left( \frac{|\lambda_{n_i}|}{|n_i|_2} \right)^{\frac{p}{p-1}} > N c^{\frac{p}{p-1}} .
$$
On the other hand the sequence $n_1, \ldots, n_N$ satisfies assumptions of  Lemma \ref{pre-krok2}. Therefore
$$
    \sum_{i=1}^N \left( \frac{|\lambda_{n_i}|}{|n_i|_2} \right)^{\frac{p}{p-1}} \leq K ,
$$
where $K$ is independent of $N$. Hence for any $N\in \mathbb{N}$ we have $N c^{\frac{p}{p-1}} < K$. This is a contradiction.
\end{proof}

\paragraph{\textit{Proof of Lemma \ref{krok2}}}

Let $\mu_k = \max_{n \in R_k} \frac{| \lambda_n |}{|n|_2}$. From Lemma \ref{lema2} we deduce the existence of a bijection $\sigma \colon \mathbb{N} \to \mathbb{N}$ such that $\mu_{\sigma(k)}$ is a non-increasing sequence. 
It is enough to show that
\begin{equation}\label{r1}
    \sum_{j=1}^{N^{d+1}} \mu_{\sigma(j)}^{\frac{p}{p-1}} < C \cdot N^d .
\end{equation}
Indeed assuming \eqref{r1}, for large enough $N$, 
\begin{equation*}
\begin{split}
    \sum_{j=1}^N \mu^{\frac{p}{p-1}}_{\sigma(j)}  \leq 
    \sum_{j=1}^{\left \lceil \sqrt[d+1]{N} \right \rceil^{d+1}} \mu^{\frac{p}{p-1}}_{\sigma(j)} 
     < C \cdot \left \lceil \sqrt[d+1]{N} \right \rceil^{d} 
     \leq C' \cdot N^{\frac{d}{d+1}}.
\end{split}    
\end{equation*}
Hence the assumptions of Lemma \ref{lema1} are satisfied and Lemma \ref{krok2} follows.

To obtain \eqref{r1} we fix $N$. Let $n_k \in R_k$ be such that $\mu_k = \frac{|\lambda_{n_k}|}{|n_k|_2}$ for $k \in \mathbb{N}$.
We divide $\mathbb{Z}^d\backslash\{0\}$ into $N$-sectors. We consider the sequence $n_{\sigma(1)}, \ldots, n_{\sigma(N^{d+1})}$. Let $S$ denote the set of all N-sectors. For $A\in S$  we denote by $I_A\subset \{1,2,\ldots, N^{d+1}\}$ the set of indices $k$ such that $n_{\sigma(k)}\in A$. We can divide the set $\{n_{\sigma(i)}: i\in I_A\}$ into at most $\left(\frac{\# I_A}{N}+ 2N+1\right)$ different $N$-sparse sequences of length $N$ (note that $\{n_{\sigma(i)}: i\in I_A\}$ can be ordered in such a way that every element is at least three times bigger than its predecessor). From Lemma \ref{pre-krok2} we get
$$
    \sum_{i \in I_A} \left( \frac{|\lambda_{n_i}|}{|n_i|_2} \right)^{\frac{p}{p-1}} \leq K \cdot \left( \frac{\#I_A}{N}  + 2N+1 \right).
$$
Summing the above inequality over all N-sectors we get 
\[
    \sum_{j=1}^{N^{d+1}} \mu_{\sigma(j)}^{\frac{p}{p-1}} \leq
    \sum_{A \in S} K \cdot \left( \frac{\# I_A}{N} + 2N +1  \right) .
\]
Observe that $\#S=d \cdot N^{d-1}$ and  $\sum_{A \in S} \# I_s = N^{d+1}$. In conclusion we get
\begin{equation*}
\begin{split}
    \sum_{A \in S} K \cdot \left( \frac{\#I_A}{N} + 2N+1 \right) & =
    K \cdot \frac{1}{N} \left(\sum_{A \in S} \#I_s \right)  + K \cdot \sum_{A \in S} \left( 1+2N \right) \\
    & = K \cdot N^{d} + K \cdot (d\cdot N^{d} + d\cdot N^{d-1}) \\
    & \leq C N^d . 
\end{split}    
\end{equation*}

\bibliography{plik2}{}

\begin{thebibliography}{10}

\bibitem{Blei14}
Ron Blei.
\newblock The {G}rothendieck inequality revisited.
\newblock {\em Mem. Amer. Math. Soc.}, 232(1093):vi+90, 2014.

\bibitem{Braverman2011}
Mark Braverman, Konstantin Makarychev, Yury Makarychev, and Assaf Naor.
\newblock The Grothendieck Constant is strictly smaller than Krivine's bound.
\newblock In {\em 2011 {IEEE} 52nd Annual Symposium on Foundations of Computer
  Science}. {IEEE}, oct 2011.

\bibitem{Pelczynski1992}
A.~Pe\l czy\'{n}ski and M.~Wojciechowski.
\newblock Paley projections on anisotropic Sobolev spaces on tori.
\newblock {\em Proceedings of the London Mathematical Society},
  s3-65(2):405--422, sep 1992.

\bibitem{Diestel}
Joe Diestel, Hans Jarchow, and Andrew Tonge.
\newblock {\em Absolutely summing operators}, volume~43 of {\em Cambridge
  Studies in Advanced Mathematics}.
\newblock Cambridge University Press, Cambridge, 1995.

\bibitem{Garling2015}
D.~J.~H. Garling.
\newblock {\em Inequalities}.
\newblock Cambridge University Press, 2015.

\bibitem{MR94682}
A.~Grothendieck.
\newblock R\'{e}sum\'{e} de la th\'{e}orie m\'{e}trique des produits tensoriels
  topologiques.
\newblock {\em Boletim da Sociedade de Matem\'{a}tica de S\~{a}o Paulo},
  8:1--79, 1953.

\bibitem{Hormander}
Lars H\"{o}rmander.
\newblock Estimates for translation invariant operators in $L_p$ spaces.
\newblock {\em Acta Mathematica}, 104(1-2):93--140, 1960.

\bibitem{Kis75}
S.~V. Kisljakov.
\newblock Sobolev imbedding operators, and the nonisomorphism of certain
  {B}anach spaces.
\newblock {\em Funkcional. Anal. i Prilo\v{z}en.}, 9(4):22--27, 1975.

\bibitem{LUST07}
Fran\c{c}oise Lust-Piquard and Quanhua Xu.
\newblock The little {G}rothendieck theorem and {K}hintchine inequalities for
  symmetric spaces of measurable operators.
\newblock {\em J. Funct. Anal.}, 244(2):488--503, 2007.

\bibitem{MR3266999}
Assaf Naor and Oded Regev.
\newblock Krivine schemes are optimal.
\newblock {\em Proc. Amer. Math. Soc.}, 142(12):4315--4320, 2014.

\bibitem{LINPEL}
J.~Lindenstrauss;~A. Pelczynski.
\newblock Absolutely summing operators in $L^p$-spaces and their applications.
\newblock {\em Studia Math}, 29:275–326, 1968.

\bibitem{Pisier2012}
Gilles Pisier.
\newblock Grothendieck's theorem, past and present.
\newblock {\em Bulletin of the American Mathematical Society}, 49(2):237--323,
  may 2012.

\bibitem{Wojciechowski1997}
M.~Wojciechowski.
\newblock On the summing property of the Sobolev embedding operators.
\newblock {\em Positivity}, 1(2):165--169, 1997.

\bibitem{Wojtaszczyk1991}
P.~Wojtaszczyk.
\newblock {\em Banach Spaces for Analysts}.
\newblock Cambridge University Press, feb 1991.

\end{thebibliography}
\bibliographystyle{plain}
\end{document}